\documentclass[10pt]{amsart}
\usepackage{amsmath}
\usepackage[usenames,dvipsnames]{color}
\usepackage{parskip}
\usepackage{amsfonts}
\usepackage{amscd}
\usepackage[centertags]{amsmath}
\usepackage{amssymb}
\usepackage{stmaryrd}
\usepackage[all,cmtip]{xy}
\usepackage[english]{babel}
\usepackage{tabularx}
\usepackage{amsxtra}
\usepackage{euscript}
\usepackage[T1]{fontenc}
\usepackage{doc, exscale, fontenc, latexsym, syntonly}
\usepackage{amsfonts}
\usepackage{amsthm}
\usepackage{graphicx}
\usepackage{mciteplus}
\usepackage{cite}

\newtheorem{theorem}{Theorem}[section]
\newtheorem{corollary}[theorem]{Corollary}

\newtheorem{proposition}[theorem]{Proposition}
\newtheorem{remark}[theorem]{Remark}
\theoremstyle{definition}
\newtheorem{definition}[theorem]{Definition}
\newtheorem{example}[theorem]{Example}
\theoremstyle{remark}

\numberwithin{figure}{section}
\numberwithin{table}{section}
\newcommand*\acknowledgment[1]{%
	\begingroup\noindent
	\rightskip\leftskip
	\begin{flushleft}\textbf{\large Acknowledgment.}\, #1%
		\par\vspace*{1mm}\end{flushleft}\endgroup}
\begin{document}

\title[DIFFERENT TYPES OF TOPOLOGICAL COMPLEXITY ON HIGHER HOMOTOPIC DISTANCE]{DIFFERENT TYPES OF TOPOLOGICAL COMPLEXITY ON HIGHER HOMOTOPIC DISTANCE}

\author{MEL\.{I}H \.{I}S and \.{I}SMET KARACA}
\date{\today}

\address{\textsc{Melih Is,}
Ege University\\
Faculty of Sciences\\
Department of Mathematics\\
Izmir, Turkey}
\email{melih.is@ege.edu.tr}
\address{\textsc{Ismet Karaca,}
Ege University\\
Faculty of Science\\
Department of Mathematics\\
Izmir, Turkey}
\email{ismet.karaca@ege.edu.tr}



\begin{abstract}
 We first study the higher version of the relative topological complexity by using the homotopic distance. We also introduced the generalized version of the relative topological complexity of a topological pair on both the Schwarz genus and the homotopic distance. With these concepts, we give some inequalities including the topological complexity and the Lusternik-Schnirelmann category, the most important parts of the study of robot motion planning in topology. Finally, by defining the parametrised topological complexity via the homotopic distance, we present some estimates on the higher setting of this concept. 
\end{abstract}

\maketitle

%%%%%%%%%%%%%%%%%%%%%%%%%%%%%%%%%%%%%%%%%%%%%%%%%%%%%
\section{Introduction}
\label{intro}
\quad Studies on determining the topological complexity TC first start with M. Farber \cite{Farber:2003}. As Farber \cite{Farber:2008} states, one of the basic methods followed in topological complexity studies is to study the different versions of the number TC and obtain a relationship between these versions and TC. This generally leads to some natural bounds for TC. For instance, if $X$ is a path-connected topological space and $Y$ is a subset of $X \times X$, then the relative topological complexity TC$_{X}(Y)$ satisfies the following inequality \cite{Farber:2008}:
\begin{eqnarray*}
	\text{TC}_{X}(Y) \leq \text{TC}(X). 
\end{eqnarray*}
Along with the definition of higher topological complexity TC$_{n}$ by Y. Rudyak \cite{Rudyak:2010}, introducing the improved version of relative topological complexity also presents a natural lower bound for TC$_{n}$. In Section \ref{sec:2}, we first focus on these higher settings of the relative topological complexity TC$_{n,X}(Y)$ via relative homotopic distance, where $Y$ is a subset of $X^{n}$.

\quad Let $(A,B)$ be a pair of topological spaces with the condition $B \subseteq A$. Then the relative topological complexity of a pair TC$(A,B)$ is defined by R. Short \cite{Short:2018} on the notion of Schwarz genus and admits the following fact:
\begin{eqnarray*}
	 \text{TC}(A,B) = \text{TC}_{A}(A \times B). 
\end{eqnarray*}
In other words, if one rewrites TC$(A,B)$ as TC$_{n,A}(B)$ (for $n=2$), then one can define the relative higher topological complexity of a pair TC$_{n}(A,B)$ by using the relative homotopic distance such as the relative higher topological complexity of a space. Before such a definition on the homotopic distance, in Section \ref{sec:3}, we first introduce TC$_{n}(A,B)$ in terms of the Schwarz genus.

\quad In recent years, topological robotics studies have focused on computing the new fiber homotopy equivalent invariant parametrised topological complexity \cite{CohenFarberWein:2020,CohenFarberWein:2021,Calcines:2022,Tanaka:2021}. For the Farber's topological complexity TC, one considers the motion planning problem as follows \cite{Farber:2003}. 
\begin{eqnarray*}
	\pi : PX \rightarrow X \times X, \ \ \pi(\gamma) = (\gamma(0),\gamma(1)),
\end{eqnarray*}
is a fibration, where $PX$ contains all continuous paths on the path-connected topological space $X$. $PX$ has a compact-open topology. TC$(X)$ is the Schwarz genus of $\pi$ and TC$(X) = n$ means that the motion planning algorithm $s_{j} : A_{j} \subset X \times X \rightarrow PX$ must be continuous with the condition $\pi \circ s_{j} = 1_{A_{j}}$ for each $j \in \{1,\cdots,n\}$. 

\quad In the parametrised version, there are some basic changes, but the main idea is familiar. This is to take the endpoints of the space and form a continuous motion between these endpoints. One has some extra conditions. The first is that the initial and final states are located in the same fiber of the fibration $q : E \rightarrow B$. The other one is to restrict the motion planning algorithm to the same fiber. We shall explain this idea mathematically. Let $q : E \rightarrow B$ be a fibration with a nonempty, path-connected fibre $X = q^{-1}(b)$ for any $b \in B$. Let $E^{I}_{B}$ denote the subspace of $E$ that consists of all paths $\alpha$ in $E$ such that $q \circ \alpha$ is the constant map. $E \times_{B} E$ is a subset of $E \times E$ that includes the point $(e_{1},e_{2})$ for which $q(e_{1})$ equals $q(e_{2})$, i.e., the pair $(e_{1},e_{2})$ is located in the same fiber. Then we consider the fibration $\pi^{'} : E^{I}_{B} \rightarrow E \times_{B} E$ with $\pi^{'}(\alpha) = (\alpha(0),\alpha(1))$. Note that the path-connected fiber of this fibration is $\Omega X$. The parametrised topological complexity TC$[q : E \rightarrow B]$ is the Schwarz genus of $\pi^{'}$ \cite{CohenFarberWein:2021}.

\quad In the sense of robot motion planning problems, the parametrised topological complexity gives an extra meaning to the base space $B$ of the fibration $q : E \rightarrow B$. The external conditions in the system can be parameterized by means of the space $B$. With the algebraic topology (especially homotopy) tools, the parametrised topological complexity determines the degree of navigational complexity as a positive number for the system when the initial and final states have the same external conditions. We first rewrite the definition of the parametrised topological complexity with respect to the homotopic distance and then we update this definition for $n > 1$ to exhibit the parametrised higher topological complexity TC$_{n}[q : E \rightarrow B]$ in Section \ref{sec:4}. 

\quad Topological complexity number varies such as relative, symmetric, monoidal, parametric etc., and accordingly, describing the higher version of all these numbers is a requirement for understanding the general concept of robot motion planning algorithms. For instance, see \cite{Gonzalez:2018} and \cite{Borat:2020} for the simplicial complexity. Davis investigates the symmetric complexity of a circle \cite{Davis:2017} and the geodesic complexity of Klein bottles with the dimension $n$ \cite{DavisRecio:2021}, respectively. In this study, we investigate the general setting of certain topological complexities. First, we recall some definitions and facts about TC and the related invariants such as cat (Lusternik-Schnirelmann category, see for a detail information \cite{CorneaLupOpTan:2003}), Schwarz genus, and the homotopic distance. Then we give the homotopic distance definitions of relative higher topological complexity of a space, relative higher topological complexity of a pair, and the parametrised higher topological complexity of a fibration in the respective sections. We also mention on some equalities and inequalities including TC$_{n}$, cat, and TC of a fibration with considering particular homotopy facts. In addition, we obtain lower and upper bounds for Hopf fibration and examine the Stiefel and Grasmann manifolds in the sense of parametrised topological complexity. 

\section{Preliminaries}
\label{sec:1}
\quad This section is dedicated to providing brief information about the different types of topological complexities and their related invariants. Note that we frequently use these facts in the following sections.

\subsection{Schwarz Genus and Homotopic Distance}
\label{subsec:1}

\begin{definition}\cite{Schwarz:1966}
	Let $q : E \rightarrow B$ be a fibration. Assume that $B$ has an open cover $\{B_{1},B_{2},\cdots,B_{r}\}$ for the minimum possible positive integer $r$ such that there is a continuous map $s_{j} : B_{j} \rightarrow E$ satsifying that $q \circ s_{j} = 1_{B_{j}}$ for each $j \in \{1,2,\cdots,r\}$. Then \textit{the Schwarz genus of $q$} is $r$. 
\end{definition}

\quad The Schwarz genus of $q$ is generally denoted by genus$(q)$ or secat$(q)$.

\begin{definition}\cite{BorVer:2021,VirgosLois:2022}
	Let $f_{j} : A \rightarrow B$ be a continuous map for each \linebreak$j \in \{1,2,\cdots,m\}$. Assume that $A$ has an open cover $\{A_{1},A_{2},\cdots,A_{r}\}$ for the minimum possible positive integer $r$ such that the condition
	\begin{eqnarray*}
		f_{1}^{i}\big|_{A_{i}} \simeq f_{2}^{i}\big|_{A_{i}} \simeq \cdots \simeq f_{m}^{i}\big|_{A_{i}}
	\end{eqnarray*}
    holds for all $i \in \{1,\cdots,r\}$. Then \textit{the higher homotopic distance of degree $m$ (or the $m-$th homotopic distance)} is $r$. 
\end{definition}

\quad The higher homotopic distance is denoted by D$(f_{1},f_{2},\cdots,f_{m})$ and is simply called \textit{the higher homotopic distance}. For $m=2$, the notion is particularly called \textit{the homotopic distance} \cite{VirgosLois:2022}.

\begin{theorem}\cite{VirgosLois:2022}\label{t2}
	Let $f_{1}$, $f_{2} : A \rightarrow B$ be two maps and $\pi : PB \rightarrow B \times B$. Assume that $p : P \rightarrow A$ is the pullback of $\pi$ by $(f_{1},f_{2}) : A \rightarrow B \times B$, where \[P = \{(a,\beta) \in A \times PB : \beta(0) = f_{1}(a) \ \text{and} \ \beta(1) = f_{2}(a)\},\] i.e., the following diagram is commutative:
	$$\xymatrix{
		P \ar[r]^{pr_{2}} \ar[d]^{p}&
		PB \ar[d]^{\pi} \\
		A \ar[r]_{(f_{1},f_{2})} & 
		B \times B.}$$
	Then D$(f_{1},f_{2}) =$ genus$(p)$.
\end{theorem}

\quad We have some useful notes on the (higher) homotopic distance for the forthcoming sections as follows \cite{BorVer:2021,VirgosLois:2022}:
\begin{proposition}\cite{BorVer:2021,VirgosLois:2022}\label{p11}
	Each of the following satisfies:
	
	\textbf{a)} If $f_{j} : A \rightarrow B$ is homotopic to $g_{j} : A \rightarrow B$ for each $j \in \{1,\cdots,n\}$, then D$(f_{1},\cdots,f_{n}) =$ D$(g_{1},\cdots,g_{n})$.
	 
	\textbf{b)} Let $1 < k < l$ and $f_{1},\cdots,f_{k},\cdots,f_{l} : A \rightarrow B$ be maps. Then we have that D$(f_{1},\cdots,f_{k}) \leq$ D$(f_{1},\cdots,f_{l})$.
	
	\textbf{c)} Let $f_{j} : A \rightarrow B$ and $p_{j} : B \rightarrow C$ be maps with $p_{j-1} \simeq p_{j}$ for each \linebreak$j \in \{2,\cdots,n\}$. Then D$(p_{1} \circ f_{1},\cdots,p_{n} \circ f_{n}) \leq$ D$(f_{1},\cdots,f_{n})$.
	
	\textbf{d)} Let $f_{j} : A \rightarrow B$ and $p_{j} : C \rightarrow A$ be maps with $p_{j-1} \simeq p_{j}$ for each \linebreak$j \in \{2,\cdots,n\}$. Then D$(f_{1} \circ p_{1},\cdots,f_{n} \circ p_{n}) \leq$ D$(f_{1},\cdots,f_{n})$.
	
	\textbf{e)} The higher homotopic distance is a homotopy invariant.
	
	\textbf{f)} Let $g$, $g^{'} : C \rightarrow A$ and $f_{1}$, $f_{2} : A \rightarrow B$ be maps. If $f_{1} \circ g^{'} \simeq f_{2} \circ g^{'}$, then D$(f_{1} \circ g,f_{2} \circ g) \leq$ D$(g,g^{'})$.
\end{proposition} 

\subsection{Topological Complexity and LS-Category}
\label{subsec:2}

\quad Topological complexity, the main part of the studies of topological robotics, can be expressed in two ways. For simplicity, we use SG and HD as abbreviations for any definition in the sense of Schwarz genus and homotopic distance, respectively.
\begin{definition}\cite{Farber:2003,VirgosLois:2022}
	Let $X$ be a path-connected topological space.
	\begin{itemize}
		\item \textbf{(SG)} Let $\pi : PX \rightarrow X \times X$, $\pi(\alpha) = (\alpha(0),\alpha(1))$, be a path fibration. Then \textit{the topological complexity of $X$} is the genus$(\pi)$.
		\item \textbf{(HD)} Let $p_{i} : X^{2} \rightarrow X$ be a projection map for each $i \in \{1,2\}$. Then \textit{the topological complexity of $X$} is D$(p_{1},p_{2})$.
	\end{itemize}
\end{definition}

\quad TC of a contractible space is $1$ and the converse is also true, that is, if TC$(X) = 1$, then $X$ is a contractible space \cite{Farber:2003}. 

\begin{proposition}\cite{Farber:2003}\label{p12}
	Let $X$ be a paracompact and locally contractible space with dim$X = n$. Then  TC$(X) \leq 2n+1$. 
\end{proposition}

\quad An important result of Proposition \ref{p12} is that $2$cat$(X)-1$ is an upper bound for TC$(X)$ when $X$ is paracompact. cat is also a natural lower bound TC. Thus, we conclude that cat$(X) \leq$ TC $\leq$ $2$cat$(X)-1$.

\begin{theorem}\label{t4}\cite{Grant:2012}
	The topological complexity of the complex Stiefel manifold $V_{r}(\mathbb{C}^{k})$, denoted by TC$(V_{r}(\mathbb{R}^{k}))$, is less than or equal to $2r(k-r)+1$.
\end{theorem}

\quad The topological complexity is denoted by TC and generalized for $n>1$ as follows.

\begin{definition}\cite{Rudyak:2010,BorVer:2021,VirgosLois:2022}
	Let $X$ be a path-connected space.
	\begin{itemize}
		\item \textbf{(SG)} Let $e_{n} : P_{n}(X) \rightarrow X^{n}$, $e_{n}(\gamma) = (\gamma_{1}(1),\gamma_{2}(1),\cdots,\gamma_{n}(1))$, be a fibration, where $P_{n}(X)$ is the space of all multipaths in $X$ such that the inital point of all paths in each multipath is the same. Then \textit{the $n-$th topological complexity (simply called the higher topological complexity) of $X$} is the genus$(e_{n})$.
		\item \textbf{(HD)} Let $p_{i} : X^{n} \rightarrow X$ be a projection map for each $i \in \{1,2,\cdots,n\}$. Then \textit{the higher topological complexity of $X$} is D$(p_{1},p_{2},\cdots,p_{n})$.
	\end{itemize}
\end{definition}

\quad The higher topological complexity is denoted by TC$_{n}$ and it is useful for $n>1$ because TC$_{1}(X) = 1$ for any space $X$. One of the important observations on TC$_{n}$ is the equality TC$_{2}(X) =$ TC$(X)$. Another is the inequality TC$_{n}(X) \leq$ TC$_{n+1}(X)$.

\quad The Lusternik Schnirelmann category (simply called LS-category or denoted by cat) is another important homotopy invariant that inspired TC. It is also a natural bound for the topological complexity number.

\begin{definition}\cite{CorneaLupOpTan:2003,VirgosLois:2022}
	Let $X$ be a path-connected topological space
	\begin{itemize}
		\item \textbf{(SG)} Let $\eta : P_{0}X \rightarrow X$, $\eta(\gamma) = (\gamma(1))$, be a Serre Fibration, where $P_{0}(X)$ is the special case of $P_{n}(X)$, i.e., the space of all paths in $X$ for which the first point of these paths is a constant point $x^{'}$ in $X$. Then \textit{the LS-category of $X$} is genus$(\eta)$.
		\item \textbf{(HD)} Let $x^{'}$ be any point of $X$. For the continuous maps $i_{1} : X \rightarrow X \times X$, $i_{1}(x) = (x,x^{'})$ and $i_{2} : X \rightarrow X \times X$, $i_{2}(x) = (x^{'},x)$, \textit{the LS-category of $X$} is D$(i_{1},i_{2})$.
	\end{itemize}
\end{definition}

\begin{proposition}\cite{James:1978}\label{p6}
	Let $A$ and $B$ be paracompact and path-connected spaces. Then cat$(A \times B) \leq$ cat$(A) +$ cat$(B)$.
\end{proposition}

\quad Assume that $q : E \rightarrow B$ is a fibration with the fiber $X$. Then we have that TC$(E) \leq$ TC$(X) +$ cat$(B \times B) + 1$\cite{NaskarSarkar:2020}. Another important observation is that the numbers TC and cat coincide when the topological space is a connected Lie group \cite{Farber:2004}.

\begin{theorem}\label{t3}\cite{Berstein:1976}
	The LS-category of the complex Grassmannian $G_{r}(\mathbb{C}^{k})$, denoted by cat$(G_{r}(\mathbb{C}^{k}))$, equals $rk$.
\end{theorem}

\begin{theorem}\label{t5}\cite{VirgosSaezTanre:2017}
	The LS-category of the quaternionic Grassmannian $G_{r}(\mathbb{H}^{k})$, denoted by cat$(G_{r}(\mathbb{H}^{k}))$, equals $k(r-k)$.
\end{theorem}

\quad It is possible to compute TC or TC$_{n}$ of a fibration as well as TC of a path-connected topological space \cite{Pavesic:2019,VirgosLoisSaez:2021,IsKaraca:2022}. 

\begin{definition} \cite{Pavesic:2019,VirgosLoisSaez:2021,IsKaraca:2022}
	Let $q : E \rightarrow B$ be a surjective fibration.
	\begin{itemize}
		\item \textbf{(SG)} For the fibration $\pi^{''} : PE \rightarrow E \times B$ defined as $\pi^{''}(\gamma) = (\gamma(0),q \circ \gamma(1))$, \textit{the topological complexity of $q$}, denoted by TC$(q)$, is genus$(\pi^{''})$.
		\item \textbf{(HD)} Let $\pi_{1} : E \times B \rightarrow E$ be the first projection map, and $\pi_{2} : E \times B \rightarrow B$ the second projection map. Then \textit{the topological complexity of $q$} is D$(\pi_{1},\pi_{2})$.
	\end{itemize}	
\end{definition}

\begin{proposition}\label{p7}\cite{Pavesic:2019}
	Let $q_{1} : E_{1} \rightarrow B_{1}$ and $q_{2} : E_{2} \rightarrow B_{2}$ be two fibrations. Then \begin{eqnarray*}
		\max\{\text{TC}(q_{1}),\text{TC}(q_{2})\} \leq \text{TC}(q_{1} \times q_{2}) \leq \text{TC}(q_{1}) + \text{TC}(q_{2})-1.
	\end{eqnarray*}
\end{proposition}

\begin{proposition}\label{p9}\cite{Pavesic:2019}
	Let $q : E \rightarrow B$ be a fibration. Then 
	\begin{eqnarray*}
		\text{cat}(B) \leq \text{TC}(q) \leq \min\{\text{TC}(B),\text{cat}(E \times B)\}.
	\end{eqnarray*}
\end{proposition}

\quad The general case of TC$(q)$ is given in \cite{IsKaraca:2022} with the following definition:

\begin{definition}\cite{IsKaraca:2022}
	Let $q : (q_{1},q_{2},\cdots,q_{n}) : E \rightarrow B^{n}$ be a surjective fibration.
	\begin{itemize}
		\item \textbf{(SG)} For a fibration $e_{n}^{q} : P_{n}(E) \rightarrow Y^{n}$, \textit{the $n-$th (higher) topological complexity of $q$} is genus$(e_{n}^{q})$.
		\item \textbf{(HD)} For each projection map $p_{i} : E^{n} \rightarrow E$ with $i \in \{1,2,\cdots,n\}$, \textit{the $n$-th (higher) topological complexity of  $q$} is D$(q \circ p_{1},q \circ p_{2},\cdots,q \circ p_{n})$.
	\end{itemize}	 
\end{definition}

\begin{theorem}\cite{IsKaraca:2022}\label{t1}
	If $q_{1} \simeq q_{2} : E \rightarrow B^{n}$ are homotopic fibrations, then 
	\[\text{TC}_{n}(q_{1}) = \text{TC}_{n}(q_{2}).\]
\end{theorem}

\begin{theorem}\label{t6}\cite{BasabeGonzalezRudyakTamaki:2014}
	Let $X$ be a path-connected space. Then
	\begin{eqnarray*}
		\text{cat}(X^{n-1}) \leq \text{TC}_{n}(X) \leq \text{cat}(X^{n}).
	\end{eqnarray*}
\end{theorem}

\subsection{Different Types of Topological Complexity}
\label{subsec:3}

\quad First, we recall the notions the relative (subspace) homotopic distance and the relative topological complexity. Later, we remind the definition of the parametrised topological complexity.

\begin{definition}\cite{VirgosLoisSaez:2021}
	Let $f$, $g : A \rightarrow B$ be two continuous maps with the subset $Y \subset A$. Then \textit{the relative homotopic distance (or subspace distance) on $Y$} is the homotopic distance of two maps $f|_{Y}$ and $g|_{Y}$, informal saying, 
	\[\text{D}_{A}(Y;f,g) = \text{D}(f|_{Y},g|_{Y}).\]
\end{definition}

\begin{definition}\cite{Farber:2008,VirgosLoisSaez:2021}
	Let $X$ be a path-connected topological space and \linebreak$Y \subset X \times X$.
	\begin{itemize}
		\item \textbf{(SG)} Let $\pi^{Y} : P_{Y}X \rightarrow Y$, $\pi^{Y}(\alpha) = (\alpha(0),\alpha(1))$, be a path fibration, where $P_{Y}X$ is a subset of $PX$ and contains all the paths in $X$ with the property that $(\alpha(0),\alpha(1))$ lies in $Y$. Then \textit{the relative topological complexity of $X$ with respect to the subspace $Y$} is genus$(\pi^{Y})$.
		\item \textbf{(HD)} Let $p_{i} : X^{2} \rightarrow X$ be the projection map for each $i \in \{1,2\}$ and $i_{Y} : Y \rightarrow X \times X$ be the inclusion map. Then \textit{the relative topological complexity of $X$ with respect to the subspace $Y$} is D$_{X \times X}(Y;p_{1},p_{2})$.
	\end{itemize}
\end{definition}

\quad The relative topological complexity of $X$ with respect to the subspace $Y$ is denoted by TC$_{Y}(X)$. 

\begin{definition}\cite{Short:2018}
	\textbf{(SG)} Let $A$ be a path-connected topological space and $B \subset A$.  Let $\pi^{A \times B} : P_{A \times B} \rightarrow A \times B$, $\pi^{A \times B}(\gamma) = (\gamma(0),\gamma(1))$, be a fibration, where $P_{A \times B}$ is the space of all paths in $A$ with the property that $\gamma(0) \in A$ and $\gamma(1) \in B$. Then \textit{the relative topological complexity of a pair $(A,B)$} is genus$(\pi^{A \times B})$.
\end{definition}

\quad The relative topological complexity of a pair $(A,B)$ is denoted by TC$(A,B)$. We mention a similar result to the ordinary TC number:

\begin{proposition}\cite{Short:2018}
	Let $(A,B)$ be the pair. Then
	\begin{eqnarray*}
		\text{TC}(A,B) = 1 \Leftrightarrow A \ \text{is contractible}.
	\end{eqnarray*}
\end{proposition}

\begin{definition}\label{d2}\cite{CohenFarberWein:2021}
	\textbf{(SG)} Let $q : E \rightarrow B$ be a fibration with nonempty, path-connected topological space $X = q^{-1}(b)$ for any $b \in B$. Let \[E^{I}_{B} = \{\beta \in PE : q \circ \beta \ \ \text{is constant}\} \subset PE\] and \[E \times_{B} E = \{(a,b) \in E \times E : q(a) = q(b)\} \subset E^{2}.\]
	Then for the fibration $\pi_{B} : E^{I}_{B} \rightarrow E \times_{B} E$ defined by $\pi_{B}(\beta) = \beta(0),\beta(1)$, \textit{the parametrised topological complexity of $q$}, denoted by TC$[q : E \rightarrow B]$, is genus$(\pi_{B})$.
\end{definition}

\begin{proposition}\label{p8}\cite{CohenFarberWein:2021}
	Let $q_{1} : E_{1} \rightarrow B_{1}$ and $q_{2} : E_{2} \rightarrow B_{2}$ be two fibrations with fibers $X_{1}$ and $X_{2}$, respectively, for metrisable spaces $E_{1}$, $E_{2}$, $B_{1}$, and $B_{2}$. For the fibration $q = q_{1} \times q_{2} : E_{1} \times E_{2} \rightarrow B_{1} \times B_{2}$ with the fiber $X_{1} \times X_{2}$, we have
	\begin{eqnarray*}
		\text{TC}[q : E_{1} \times E_{2} \rightarrow B_{1} \times B_{2}] \leq \text{TC}[q_{1} : E_{1} \rightarrow B_{1}] + \text{TC}[q_{2} : E_{2} \rightarrow B_{2}].
	\end{eqnarray*}
\end{proposition}

\section{Relative Higher Topological Complexity of a Space}
\label{sec:2}
\begin{definition}
	\textbf{(HD)} Let $X$ be a path-connected topological space and $Y \subseteq X^{n}$ be a subspace. Then the relative higher topological complexity is defined as \[\text{TC}_{n,X}(Y) = \text{D}_{X^{n}}(Y;p_{1}, \cdots, p_{n}),\] where $p_{i} : X^{n} \rightarrow X$ is a projection onto the $i-$th factor for each $i = 1, \cdots, n$.
\end{definition}

\quad In particular, if $Y$ is choosen as $X^{n}$, then we find TC$_{n,X}(X^{n}) =$ TC$_{n}(X)$. Indeed,
\begin{eqnarray*}
	\text{TC}_{n,X}(X^{n}) = \text{D}_{X^{n}}(X^{n};p_{1}, \cdots, p_{n}) = \text{D}(p_{1}, \cdots, p_{n}) = \text{TC}_{n}(X).
\end{eqnarray*}

\begin{proposition}
	Let $Y$ be a subspace of $X^{n}$ and $p_{i} : X^{n} \rightarrow X$ be the corresponding projection map for each $i = 1, \cdots, n$. Then TC$_{n,X}(Y) = 1$ if and only if the projections $(p_{i})|_{Y} : Y \rightarrow X$ with $i = 1, \cdots, n$ are homotopic to each other.
\end{proposition}

\begin{proof}
	If $Y$ is a subspace of $X^{n}$, then we get
	\begin{eqnarray*}
		\text{TC}_{n,X}(Y) = 1 \Leftrightarrow \text{D}_{X^{n}}(Y;p_{1}, \cdots, p_{n}) = 1 \Leftrightarrow (p_{1})|_{Y} \simeq \cdots \simeq (p_{n})|_{Y}.
	\end{eqnarray*}
\end{proof}

\begin{proposition}\label{p1}
	Let $Y$ be a subspace of $X^{n}$. Then TC$_{n,X}(Y) \leq$ TC$_{n}(X)$.
\end{proposition}

\begin{proof} If $Y$ is a subspace of $X^{n}$, then we have an inclusion map $i_{Y} : Y \rightarrow X^{n}$. Therefore, by (d) part of Proposition \ref{p11}, we observe that
	\begin{eqnarray*}
		\text{TC}_{n,X}(Y) &=& \text{D}_{X^{n}}(Y;p_{1}, \cdots, p_{n}) \\
		&=& \text{D}(p_{1} \circ i_{Y}, \cdots, p_{n} \circ i_{Y})\\
		&\leq& \text{D}(p_{1}, \cdots, p_{n}) \\
		&=& \text{TC}_{n}(X).
	\end{eqnarray*}
\end{proof}

\begin{proposition}\label{p3}
	Let $Y \subset Z \subset X^{n}$. Then TC$_{n,X}(Y) \leq$ TC$_{n,X}(Z)$.
\end{proposition}

\begin{proof} Let $i_{Y} : Y \rightarrow X^{n}$, $i_{Z} : Z \rightarrow X^{n}$, and $i : Y \rightarrow Z$ be three inclusion maps. Then $i_{Y}$ can be rewritten as the composition of $i$ and $i_{Z}$ i.e., $i_{Y} = i_{Z} \circ i$. From Proposition \ref{p11} (d), we have
	\begin{eqnarray*}
		\text{TC}_{n,X}(Y) &=& \text{D}_{X^{n}}(Y;p_{1},\cdots,p_{n}) 
		= \text{D}(p_{1} \circ i_{Y},\cdots,p_{n} \circ i_{Y}) \\
		&=& \text{D}(p_{1} \circ i_{Z} \circ i,\cdots,p_{n} \circ i_{Z} \circ i) 
		\leq \text{D}(p_{1} \circ i_{Z},\cdots,p_{n} \circ i_{Z}) \\
		&=& \text{D}_{X^{n}}(Z;p_{1},\cdots,p_{n}) \\
		&=& \text{TC}_{n,X}(Z).
	\end{eqnarray*}
\end{proof}

\begin{proposition}\label{p2}
	If $Y$ is a retract of $X$, then TC$_{n,X}(Y^{n}) \geq$ TC$_{n}(Y)$.
\end{proposition}

\begin{proof}
	Let $r : X \rightarrow Y$ be a retraction map. Let $p_{j} : X^{n} \rightarrow X$ and $q_{j} : Y^{n} \rightarrow Y$ be the projection maps for each $j = 1, \cdots, n$. By (c) part of Proposition \ref{p11}, we find
	\begin{eqnarray*}
		\text{TC}_{n,X}(Y^{n}) &=& \text{D}_{X^{n}}(Y^{n};p_{1},\cdots,p_{n})
		= \text{D}(p_{1} \circ i_{Y^{n}},\cdots,p_{n} \circ i_{Y^{n}}) \\
		&\geq& \text{D}(r \circ p_{1} \circ i_{Y^{n}},\cdots,r \circ p_{n} \circ i_{Y^{n}}) 
		= \text{D}(q_{1},\cdots,q_{n}) \\
		&=& \text{TC}_{n}(Y)
	\end{eqnarray*} 
    with considering the fact that $q_{j} = r \circ p_{j} \circ i_{Y^{n}}$ for each $j$.
\end{proof}

\begin{corollary}
	If $Y$ is a retract of $X$ then, TC$_{n}(Y) \leq$ TC$_{n}(X)$.
\end{corollary}

\begin{proof}
	Since $Y \subset X$ is a retract, $Y^{n} \subset X^{n}$. Then Proposition \ref{p3} says that TC$_{n,X}(Y^{n}) \leq$ TC$_{n,X}(X^{n}) =$ TC$_{n}(X)$. Using Proposition \ref{p2}, we conclude that TC$_{n}(Y) \leq$ TC$_{n}(X)$.
\end{proof}

\begin{proposition}
	Let $\{Y_{1}, \cdots, Y_{m}\}$ be an open covering of $X^{n}$. Then \[\text{TC}_{n,X}(Y_{1}) + \cdots + \text{TC}_{n,X}(Y_{m}) \leq m \cdot \text{TC}_{n}(X).\]
\end{proposition}

\begin{proof} Let $Y_{1}, \cdots, Y_{m}$ be an open subsets of $X^{n}$ with the fact $X^{n} = Y_{1} \cup \cdots \cup Y_{m}$. Then we have that
	\begin{eqnarray*}
		&&TC_{n,X}(Y_{1}) + \cdots + TC_{n,X}(Y_{m}) \\
		&=&D_{X^{n}}(Y_{1};p_{1},\cdots,p_{n}) + \cdots + D_{X^{n}}(Y_{m};p_{1},\cdots,p_{n}) \\
		&=&
		D(p_{1} \circ i_{Y_{1}}, \cdots, p_{n} \circ i_{Y_{1}}) + \cdots + D(p_{1} \circ i_{Y_{m}}, \cdots, p_{n} \circ i_{Y_{m}})
		\\
		&\leq&
		D(p_{1},\cdots,p_{n}) + \cdots + D(p_{1},\cdots,p_{n})
		\\
		&=&m \cdot TC_{n}(X)  
	\end{eqnarray*}
    from (d) part of Proposition \ref{p11}.
\end{proof}

\begin{theorem}\label{p4}
	Let $Y \subseteq X^{2}$. Then TC$_{2,X}(Y) \leq$ cat$_{X \times X}(Y)$.
\end{theorem}

\begin{proof}
	Let $p_{1},p_{2} : X \times X \rightarrow X$ be projections and $\ast : Y \rightarrow X \times X$ a constant map. The fact $p_{1} \circ \ast \simeq p_{2} \circ \ast$ implies D$(p_{1} \circ i_{Y},p_{2} \circ i_{Y}) \leq$ D$(i_{Y},\ast)$ from (f) of Proposition \ref{p11}. Thus, we conclude that TC$_{2,X}(Y) \leq$ cat$_{X \times X}(Y)$.
\end{proof}

\quad It is clear that TC$_{2,X}(Y) =$ TC$_{X}(Y)$. The previous theorem also confirms the fact that TC$_{X}(Y) \leq$ cat$_{X \times X}(Y)$. Theorem \ref{p4} can be generalized as follows with considering Theorem \ref{t6}:

\begin{corollary}
	Let $Y \subseteq X^{n}$. Then TC$_{n,X}(Y) \leq$ cat$_{X^{n}}(Y)$.
\end{corollary}

\begin{theorem}
	Let $Y$, $Z \subset X^{n}$ such that $Y$ and $Z$ have the same homotopy type. Then TC$_{n,X}(Y) = $ TC$_{n,X}(Z)$.
\end{theorem}

\begin{proof}
	Let $\beta : Z \rightarrow Y$ be the homotopy equivalence map. Assume that $i \in \{1,\cdots,n\}$. Consider the following commutative diagram for the projection $p_{i} : X^{n} \rightarrow X$ with the inclusions $i_{Y} : Y \rightarrow X^{n}$ and $i_{Z} : Z \rightarrow X^{n}$:
	$$\xymatrix{
		Y \ar[r]^{p_{i} \circ i_{Y}} &
		X \ar[d]^{1_{X}} \\
		Z \ar[u]^{\beta} \ar[r]_{p_{i} \circ i_{Z}} & X.}$$
	Hence, by Theorem \ref{t1}, we get 
	\begin{eqnarray*}
		\text{D}(p_{1} \circ i_{Y},\cdots,p_{n} \circ i_{Y}) = \text{D}(p_{1} \circ i_{Z},\cdots,p_{n} \circ i_{Z})
	\end{eqnarray*}
    which concludes that TC$_{n,X}(Y) = $ TC$_{n,X}(Z)$.
\end{proof}

\section{Relative Higher Topological Complexity of a Pair}
\label{sec:3}
\begin{definition}
	\textbf{(SG)} Let $A$ be a path-connected space and $B \subseteq A$. Set
	\begin{eqnarray*}
		P^{'}_{A \times B} = \{\alpha \in A^{J_{n}} \ : \ \alpha(0) \in A, \ \alpha(1) = (\alpha_{1}(1), \cdots,\alpha_{n}(1)) \in B\} \subseteq A^{J_{n}}.
	\end{eqnarray*}
    Then for a fibration $e_{n}^{'} : P^{'}_{A \times B} \rightarrow B^{n}$ defined with $e_{n}^{'}(\alpha) = (\alpha_{1}(1), \cdots,\alpha_{n}(1))$, the relative higher topological complexity of the pair $(A,B)$ is defined as
    \begin{eqnarray*}
    	TC_{n}(A,B) = genus(e_{n}^{'}).
    \end{eqnarray*}
\end{definition}

\quad Note that $e_{n}^{'} : P^{'}_{A \times B} \rightarrow B^{n}$ is indeed a fibration because the restriction of a fibration $e_{n} : A^{J_{n}} \rightarrow A^{n}$, $e_{n}(\alpha) = (\alpha_{1}(1), \cdots,\alpha_{n}(1))$, to a subset $B^{n} \subset A^{n}$ is $e_{n}^{'}$.

\begin{proposition}\label{p5}
	\textbf{a)} TC$_{1}(A,B) = 1$. This means that the notation TC$_{n}(A,B)$ is significative for $n>1$.
     
	\textbf{b)} TC$_{n}(A,B) = $ genus$(d_{n}^{'})$, where $d_{n}^{'} : B \rightarrow B^{n}$ is a diagonal map since $e_{n}^{'}$ is a fibrational substitute of $d_{n}^{'}$.
	
	\textbf{c)} For $n = 2$, TC$_{2}(A,B) = $ TC$(A,B)$. 
\end{proposition}

\begin{proof}
	\textbf{a)} Let $n=1$. Then for a fibration $e_{1}^{'} : P^{'}_{A \times B} \rightarrow B$ with $e_{1}^{'}(\alpha) = \alpha_{1}(1)$, we construct a map $s : B \rightarrow P^{'}_{A \times B}$ such that $s$ takes any point $y$ of $B$ to the constant path $\epsilon_{y}$ at this point. Therefore, we get
	\begin{eqnarray*}
		e_{1}^{'} \circ s(y) = e_{1}^{'}(\epsilon_{y}) = y = 1_{B}(y).
	\end{eqnarray*}
	Thus, genus$(e_{1}^{'})$ equals $1$.
	
	\textbf{b)} Take a homotopy equivalence $h : B \rightarrow P^{'}_{A \times B}$ defined as $h(y) = \epsilon_{y}$, where $\epsilon_{y}$ is a constant path at $y$. Then we get
	\begin{eqnarray*}
		{e}_{n}^{'} \circ h(y) = {e}_{n}^{'}(\epsilon_{y}) = (y,\cdots,y) = d_{n}^{'}(y).
	\end{eqnarray*}
	
	\textbf{c)} Let $n = 2$. Define $e_{2}^{''} : P_{A \times B} \rightarrow A \times B$ as $e_{2}^{''}(\alpha) = (\alpha(0),\alpha(1))$. Then $e_{2}^{''}$ is a fibrational substitute of the diagonal map $d_{2}^{'} : B \rightarrow B^{2}$ because $h^{'} : B \rightarrow P_{A \times B}$, $h(y) = \epsilon_{y}$, is a homotopy equivalence and the condition $e_{2}^{''} \circ h^{'} = d_{2}^{'}$ holds with considering that $B \subseteq A$. Thus, we find TC$_{2}(A,B) =$ genus$(e_{2}^{''}) =$ TC$(A,B)$.
\end{proof}

\quad One of the well-known results of Schwarz \cite{Schwarz:1966} leads to us having an important relationship between the relative higher topological complexity and the Lusternik-Schnirelmann category:

\begin{corollary}\label{c1}
	For a path-connected space $A$ with its subset $B$, we have TC$_{n}(A,B) \leq$ cat$(B^{n})$. In addition, if $A$ is contractible, then we conclude that TC$_{n}(A,B) =$ cat$(B^{n})$.
\end{corollary}

\quad By Corollary \ref{c1}, we immediately have that if $B$ is contractible, then we get \[\text{TC}_{n}(A,B) \leq \text{cat}(B^{n}) = 1,\] i.e.,  TC$_{n}(A,B) = 1$. 

\begin{example}
	For any point $x_{0}$ in a path-connected space $A$, we obtain that TC$_{n}(A,\{x_{0}\}) = 1$.
\end{example}

\quad It is possible to improve Corollary \ref{c1} with using Proposition \ref{p6}:

\begin{corollary}
	Let $A$ be path-connected and $B$ be a path-connected and paracompact subset of $A$. Then we have TC$_{n}(A,B) \leq$ $n \cdot$cat$(B)$. 
\end{corollary}

\begin{proof}
	By Corollary \ref{c1}, we obtain TC$_{n}(A,B) \leq$ cat$(B^{n})$. Since $B$ is path-connected and paracompact, we observe that cat$(B^{n}) \leq n \cdot$cat$(B)$.
\end{proof}

\quad Besides Schwarz genus, the relative higher topological complexity TC$_{n}(A,B)$ of a pair $(A,B)$ can also be defined by higher homotopic distance:

\begin{definition}
	\textbf{(HD)} Let $A$ be a path-connected space and $B \subseteq A$. Then \[\text{TC}_{n}(A,B) = \text{D}_{A^{n}}(A \times B \times B \times \cdots \times B;p_{1},p_{2},\cdots,p_{n})\] for $n>1$, where each $p_{i}$ is a projection from $A \times A \times \cdots \times A$ to $A$ onto the $i-$th factor for $i = 1, \cdots, n$. 
\end{definition}

\quad We assume that TC$_{1}(A,B)$ always equals $1$. If $B = A$, then we conclude that TC$_{n}(A,A) =$ TC$_{n}(A)$. We observe that TC$_{n}(A,B) \leq$ TC$_{n+1}(A,B)$ as well as TC$_{n}$ of a space or a fibration.

\begin{proposition}
	Let $B_{1} \subset B_{2} \subset A$. Then TC$_{n}(A,B_{1}) \leq$ TC$_{n}(A,B_{2})$.
\end{proposition}

\begin{proof}
	When we consider three inclusion maps $i_{B_{1}} : B_{1} \rightarrow A$, $i_{B_{2}} : B_{2} \rightarrow A$, $i : B_{1} \rightarrow B_{2}$ such that $i_{B_{1}} = i_{B_{2}} \circ i$, the remaining part of the proof goes similar to the proof of Proposition \ref{p3}.
\end{proof}

\begin{proposition}
	If $A$ is path-connected with a subset $B \subset A$, then we have TC$_{n}(A,B) \leq$ TC$_{n}(A)$.
\end{proposition}

\begin{proof}
	Let $i : A \times B \times B \times \cdots \times B \rightarrow A^{n}$ be an inclusion map. Then Proposition \ref{p11} (d) gives us that 
	\begin{eqnarray*}
		\text{TC}_{n}(A,B) &=& \text{D}_{A^{n}}(A \times B \times B \times \cdots \times B;p_{1},p_{2},\cdots,p_{n}) \\
		&=& \text{D}(p_{1} \circ i,p_{2} \circ i,\cdots,p_{n} \circ i) \\
		&\leq& \text{D}(p_{1},p_{2},\cdots,p_{n}) \\
		&=& \text{TC}_{n}(A).
	\end{eqnarray*}
\end{proof}

\section{Parametrised (Higher) Topological Complexity Using Homotopic Distance}
\label{sec:4}

\quad The task in this section is to mention the homotopic distance definition of the parametrised topological complexity. Let $q : E \rightarrow B$ be a fibration and $X \neq \emptyset$ is a path-connected fiber for the fibration $q$. Let $E_{B}^{I}$ be a set consisting of all continuous paths $\gamma$ in $E$ such that $q \circ \gamma$ is a constant path. Moreover, $E \times_{B} E$ is a subset of $E \times E$ and it contains all points $(e,e^{'})$ with the condition $q(e)$ equals $q(e^{'})$. Recall that the map $\pi : E_{B}^{I} \rightarrow E \times_{B} E$ with $\pi(\gamma) = (\gamma(0),\gamma(1))$ is a fibration with fibre $\Omega X$. Then consider the following diagram (see also Theorem \ref{t2}):
\[\xymatrix{
	P \ar[r]^{\pi_{2}} \ar[d]_{\pi_{1}} &
	E_{B}^{I} \ar[d]^{\pi} \\
	E \times_{B} E \ar[r]_{(p_{1}^{B},p_{2}^{B})} & E \times_{B} E}.\]
Here $p_{i}^{B} : E \times_{B} E \rightarrow E$ is a projection map onto the $i$th factor for $i = 1, 2$. Similarly, $\pi_{i}$ is another projection map for each $i = 1, 2$. Also, $P$ is given by the set $\{(e,e^{'},\gamma) \ : \ \gamma(0) = p_{1}(e,e^{'}) = e, \ \gamma(1) = p_{2}(e,e^{'}) = e^{'}\}$, and it is clearly a subset of $E\times_{B} E \times E_{B}^{I}$. This yields that $D(p_{1}^{B},p_{2}^{B}) = genus(\pi_{1})$. Since $(p_{1}^{B},p_{2}^{B}) = 1_{E \times_{B} E}$, we observe that $genus(\pi_{1}) = genus(\pi)$. Combining this result with the Definition \ref{d2}, we have a new statement of the parametrised topological complexity on homotopic distance:
\begin{definition}\label{d1}
	\textbf{(HD)} The parametrised topological complexity is defined as
	\begin{eqnarray*}
		\text{TC}[q : E \rightarrow B] = D(p_{1}^{B},p_{2}^{B})
	\end{eqnarray*}
    for the projection map $p_{i}^{B} : E \times_{B} E \rightarrow E$ with each $i = 1, 2$.
\end{definition}

\quad For the improved version of Definition \ref{d1}, we note that the parametrised higher topological complexity is defined as 
\begin{eqnarray*}
	\text{TC}_{n}[q : E \rightarrow B] = D(p_{1}^{B}, \cdots,p_{n}^{B})
\end{eqnarray*}
for $n>1$ and the map $p_{i}^{B} : E \times_{B} E \cdots \times_{B} E \rightarrow E$ with each $i = 1, \cdots, n$.

\quad If $n = 1$, then TC$_{1}[q : E \rightarrow B]$ is always $1$. The second observation states that the inequality TC$_{n}[q : E \rightarrow B] \leq$ TC$_{n+1}[q : E \rightarrow B]$ holds. Furthermore, one can easily observe the equality TC$_{2}[q : E \rightarrow B] =$ TC$[q : E \rightarrow B]$ by using the higher homotopic distance.

\begin{proposition}\label{p10}
	Let $q : E \rightarrow B$ be a fibration. Let $q|_{B^{'}} = q^{'} : E^{'} \rightarrow B^{'}$ be another fibration with $B^{'} \subset B$ and $E^{'} = q^{-1}(B^{'})$. Then 
	\begin{eqnarray*}
		\text{TC}[q^{'} : E^{'} \rightarrow B^{'}] \leq \text{TC}[q : E \rightarrow B].
	\end{eqnarray*} 
\end{proposition}

\begin{remark}
	Proposition \ref{p10} is first expressed in \cite{CohenFarberWein:2021}. We now give explicit proof by using the homotopic distance.
\end{remark}

\begin{proof}
	By Definition \ref{d1}, TC$[q : E \rightarrow B]$ and TC$[q^{'} : E^{'} \rightarrow B^{'}]$ are respectively equal to D$(p_{1}^{B},p_{2}^{B})$ and D$(p_{1}^{B^{'}},p_{2}^{B^{'}})$ for the projection maps $p_{i}^{B} : E \times_{B} E \rightarrow E$ and $p_{i}^{B^{'}} : E^{'} \times_{B^{'}} E^{'} \rightarrow E^{'}$ for each $i = 1, 2$. The projection map $p_{i}^{B^{'}}$ can be thought as $p_{i}^{B} \circ j$, where $j : E^{'} \times_{B^{'}} E^{'} \hookrightarrow E \times_{B} E$ is an inclusion. Therefore, we have
	\begin{eqnarray*}
		D(p_{1}^{B^{'}},p_{2}^{B^{'}}) = D(p_{1}^{B} \circ j,p_{2}^{B} \circ j)
	\end{eqnarray*}
    via (d) of Proposition \ref{p11}. Finally, the inequality D$(p_{1}^{B} \circ j,p_{2}^{B} \circ j) \leq$ D$(p_{1}^{B},p_{2}^{B})$ concludes that TC$[q^{'} : E^{'} \rightarrow B^{'}] \leq$ TC$[q : E \rightarrow B]$.
\end{proof}

\quad It is possible that the parametrised topological complexity can be defined by the relative topological complexity. Indeed,
\begin{eqnarray*}
	\text{TC}[q : E \rightarrow B] &=& \text{D}(p_{1}^{B},p_{2}^{B}) \\
	&=& \text{D}(p_{1} \circ i_{E \times_{B} E},p_{2} \circ i_{E \times_{B} E})\\
	&=& \text{D}_{E^{2}}(E \times_{B} E;p_{1},p_{2})\\
	&=& \text{TC}_{E}(E \times_{B} E).
\end{eqnarray*}

This fact is also improved with the following equality: \[\text{TC}_{n}[q : E \rightarrow B] = \text{TC}_{n,E}(E \times_{B} E \cdots \times_{B} E).\]

\begin{proposition}
	TC$_{n}[q : E \rightarrow B] \leq$ TC$_{n}(E)$.
\end{proposition}

\begin{proof}
     Let $j : E \times_{B} E \cdots \times_{B} E \rightarrow E \times E \cdots \times E$ be the inclusion map. Then, by Proposition \ref{p11} (d), we have
     \begin{eqnarray*}
     	\text{TC}_{n}[q : E \rightarrow B] &=& \text{D}(p_{1}^{B},\cdots,p_{n}^{B}) \\
     	&=& \text{D}(p_{1} \circ j,\cdots,p_{n} \circ j) \\
     	&\leq& \text{D}(p_{1},\cdots,p_{n}) = \text{TC}_{n}(E).
     \end{eqnarray*}
\end{proof}

\begin{corollary}\label{c3}
	TC$[q : E \rightarrow B] \leq$ TC$(E)$.
\end{corollary}

\begin{theorem}
	TC$_{n}(q) \leq$ TC$_{n}[q : E \rightarrow B]$.
\end{theorem}

\begin{proof}
	Let $j : E \times_{B} E \cdots \times_{B} E \rightarrow E \times E \cdots \times E$ be an inclusion map and \linebreak $p_{i} : E^{n} \rightarrow E$ be the projection map for each $i = 1, \cdots, n$. Then we find
	\begin{eqnarray*}
		\text{TC}_{n}(q) &=& \text{D}(q \circ p_{1},\cdots,q \circ p_{n}) \\
		&\leq& \text{D}(p_{1},\cdots,p_{n})\\
		&\leq& \text{D}(p_{1} \circ j,\cdots,p_{n} \circ j)\\
		&=& \text{D}(p_{1}^{B},\cdots,p_{n}^{B}) = \text{TC}_{n}[q : E \rightarrow B]
	\end{eqnarray*}
    by using (c) and (d) parts of Proposition \ref{p11}, respectively.
\end{proof}

\begin{corollary}\label{c2}
	TC$(q) \leq$ TC$[q : E \rightarrow B]$.
\end{corollary}

\begin{corollary}\label{c4}
	\textbf{a)} Let $q_{1} : E_{1} \rightarrow B_{1}$ be a fibration with a path-connected fiber $X_{1}$ and $q_{2} : E_{2} \rightarrow B_{2}$ be another fibration with a path-connected fiber $X_{2}$. If $E_{1}$, $B_{1}$, $E_{2}$ and $B_{2}$ are metrisable, then 
	\begin{eqnarray*}
		\max\{TC(q_{1}),TC(q_{2})\} \leq TC[q_{1} : E_{1} \rightarrow B_{1}] + TC[q_{2} : E_{2} \rightarrow B_{2}].
	\end{eqnarray*}
    \textbf{b)} cat$(B) \leq$ TC$[q : E \rightarrow B]$.
\end{corollary}

\begin{proof}
	\textbf{a)} Let $q_{1} \times q_{2} : E_{1} \times E_{2} \rightarrow B_{1} \times B_{2}$ be a fibration with a path-connected fiber $X_{1} \times X_{2}$. Then by Proposition \ref{p7}, Corollary \ref{c2}, and Proposition \ref{p8}, respectively, we have
	\begin{eqnarray*}
		\max\{TC(q_{1}),TC(q_{2})\} &\leq& TC(q_{1} \times q_{2}) \\ 
		&\leq& TC[q_{1} \times q_{2} : E_{1} \times E_{2} \rightarrow B_{1} \times B_{2}] \\ 
		&\leq& TC[q_{1} : E_{1} \times B_{1}] + TC[q_{2} : E_{2} \rightarrow B_{2}].
	\end{eqnarray*}
    \textbf{b)} This is obvious from Proposition \ref{p9} and Corollary \ref{c2}. 
\end{proof}

\begin{example}
	\textbf{Hopf Fibration:} Consider the Hopf fibration $q : S^{3} \rightarrow S^{2}$ with the path-connected fiber $S^{1}$. We shall show that $2 \leq$ TC$_{n}[q : S^{3} \rightarrow S^{2}] \leq n$ for $n \geq 2$. Since TC$_{n}(S^{3}) = n$ \cite{Rudyak:2010}, by Corollary \ref{c3}, TC$_{n}[q : S^{3} \rightarrow S^{2}] \leq n$. On the other hand, cat$(S^{2}) = 2$ yields that TC$_{n}[q : S^{3} \rightarrow S^{2}] \geq 2$ via (b) part of Corollary \ref{c4}.
\end{example}

\begin{example}
	\textbf{Stiefel and Grassmann Manifolds:} Consider the complex case, i.e.,
	\begin{eqnarray*}
		U(r) \longrightarrow V_{r}(\mathbb{C}^{k}) \stackrel{q_{1}}{\longrightarrow} G_{r}(\mathbb{C}^{k}).
	\end{eqnarray*}
	Since cat$(G_{r}(\mathbb{C}^{k})) = rk$ by Theorem \ref{t3}, Corollary \ref{c4} (b) states that 
	\[rk \leq \text{TC}[q_{1} : V_{r}(\mathbb{C}^{k}) \rightarrow G_{r}(\mathbb{C}^{k})].\] On the other hand, we have that TC$(V_{r}(\mathbb{C}^{k})) \leq 2r(k-r)+1$ from Theorem \ref{t4}. By Corollary \ref{c3}, we get TC$[q_{1} : V_{r}(\mathbb{C}^{k}) \rightarrow G_{r}(\mathbb{C}^{k})] \leq 2r(k-r)+1$. Finally, we conclude that
	\begin{eqnarray*}
		rk \leq \text{TC}[q_{1} : V_{r}(\mathbb{C}^{k}) \rightarrow G_{r}(\mathbb{C}^{k})] \leq 2r(k-r)+1
	\end{eqnarray*}
	for $n=2$ in the sense of TC$_{n}$.\\ 
	If we assume that the quaternionic case, namely that,
	\begin{eqnarray*}
		Sp(r) \longrightarrow V_{r}(\mathbb{H}^{k}) \stackrel{q_{2}}{\longrightarrow} G_{r}(\mathbb{H}^{k}),
	\end{eqnarray*}
	then the inequality
	\begin{eqnarray*}
		k(r-k) \leq \text{TC}[q_{2} : V_{r}(\mathbb{H}^{k}) \rightarrow G_{r}(\mathbb{H}^{k})]
	\end{eqnarray*}
	holds from the fact that $k(r-k) =$ cat$(G_{r}(\mathbb{H}^{k}))$ by Theorem \ref{t5}.\\
	Now consider the real case:
	\begin{eqnarray*}
		O(n) \longrightarrow V_{r}(\mathbb{R}^{k}) \stackrel{q_{3}}{\longrightarrow} G_{r}(\mathbb{R}^{k}).
	\end{eqnarray*}
	First, assume that $r = 2$ and $k = 2^{p}+1$ with any integer $p > 0$. Then, by Theorem 2.8 of \cite{AkhtarifarAsaidi:2020}, we get TC$[q_{3} : V_{2}(\mathbb{R}^{2^{p}+1}) \rightarrow G_{2}(\mathbb{R}^{2^{p}+1})] \geq 2^{p+1}-2$. For the upper bound, we shall use Corollary \ref{c3}. Since dim$(V_{2}(\mathbb{R}^{2^{p}+1})) = 2^{p+1}-1$, Proposition \ref{p12} says that TC$(V_{2}(\mathbb{R}^{2^{p}+1})) \leq 2^{p+2}-1$. As a consequence,
	\begin{eqnarray*}
		2^{p+1}-2 \leq \text{TC}[q_{3} : V_{2}(\mathbb{R}^{2^{p}+1}) \rightarrow G_{2}(\mathbb{R}^{2^{p}+1})] \leq 2^{p+2}-1.
	\end{eqnarray*}
	Now, assume that $r = 2$ and $k = 2^{p}+2$ with any integer $p$. Similar to the previous case, by using Theorem 2.11 in \cite{AkhtarifarAsaidi:2020}, we have the following inequalities:
	\begin{eqnarray*}
		2^{p+1}-1 \leq \text{TC}[q_{3} : V_{2}(\mathbb{R}^{2^{p}+2}) \rightarrow G_{2}(\mathbb{R}^{2^{p}+2})] \leq 2^{p+2} + 3
	\end{eqnarray*}
    with considering the fact that dim$(V_{2}(\mathbb{R}^{2^{p}+2})) = 2^{p+1}+1$.
\end{example}

\begin{theorem}
	If $q_{1} : E_{1} \rightarrow B_{1}$ and $q_{2} : E_{2} \rightarrow B_{1}$ are fiber homotopy equivalent with the same nonempty path-connected fiber $X_{1}$ for both two fibrations $q_{1}$ and $q_{2}$, then \[\text{TC}_{n}[q_{1} : E_{1} \rightarrow B_{1}] =  \text{TC}_{n}[q_{2} : E_{2} \rightarrow B_{1}],\] that is, the parametrised higher topological complexity is a fiber homotopy equivalent invariant.
\end{theorem}

\begin{proof}
	Let $q_{1} : E_{1} \rightarrow B_{1}$ and $q_{2} : E_{2} \rightarrow B_{1}$ are fiber homotopy equivalent with the fiber $X_{1}$. Then we have two maps $h : E_{1} \rightarrow E_{2}$ and $k : E_{2} \rightarrow E_{1}$ satisfying two conditions $h \circ k \simeq 1_{E_{2}}$ and $k \circ h \simeq 1_{E_{1}}$.  For simplicity, we rewrite $E_{1}^{'}$ and $E_{2}^{'}$ as $E_{1} \times_{B_{1}} E_{1} \times_{B_{1}} \cdots \times_{B_{1}} E_{1}$ and $E_{2} \times_{B_{1}} E_{2} \times_{B_{1}} \cdots \times_{B_{1}} E_{2}$, respectively. Consider the homotopy equivalence map $\beta : E_{2}^{'} \rightarrow E_{1}^{'}$. Assume that $p_{i}^{B} : E_{1}^{'} \rightarrow E_{1}$ and $q_{i}^{B} : E_{2}^{'} \rightarrow E_{2}$ are projections for each $i \in \{1,\cdots,n\}$. Then the following commutative diagram
		$$\xymatrix{
		E_{1}^{'} \ar[r]^{p_{i}^{B_{1}}} &
		E_{1} \ar[d]^{h} \\
		E_{2}^{'} \ar[u]^{\beta} \ar[r]_{q_{i}^{B_{1}}} & E_{2}}$$
	states that D$(p_{1}^{B_{1}},\cdots,p_{n}^{B_{1}}) =$ D$(q_{1}^{B_{1}},\cdots,q_{n}^{B_{1}})$. As a consequence, the equality TC$_{n}[q_{1} : E_{1} \rightarrow B_{1}] =$ TC$_{n}[q_{2} : E_{2} \rightarrow B_{1}]$ holds.
\end{proof}

\section{Conclusion}
\label{sec:5}
\quad The relative topological complexity is one of the first examples of different types of TC. With this approach, a lower bound for TC is obtained. A similar approach to TC$_{n}$ is done by the relative higher topological complexity on the notion homotopic distance in this study. Interestingly, it also gives a new way to introduce the parametrised topological complexity, which has been one of the most popular numbers among several TC versions in the last few years due to the different constructions of its motion planning algorithm. The higher homotopic distance again leads to define the parametrised higher topological complexity. These higher settings of TC always provide a different and strong perspective on the subject of robot motion planning problems in daily life. 

\acknowledgment{The first author is granted as a fellowship by the Scientific and Technological Research Council of Turkey TUBITAK-2211-A. In addition, this work was partially supported by the Research Fund of Ege University (Project Number: FDK-2020-21123).}

\end{document}